\documentclass[12pt,centertags,oneside]{amsart}
\usepackage{amsmath,amstext,amsthm,amscd,typearea,hyperref}
\usepackage{amssymb}
\usepackage{a4wide}
\usepackage[mathscr]{eucal}
\usepackage{mathrsfs}
\usepackage{typearea}
\usepackage{charter}
\usepackage{pdfsync}
\usepackage{stmaryrd}
 \usepackage[T1]{fontenc}
\usepackage{enumitem}

\usepackage{calc}

\usepackage{changepage}

\usepackage[a4paper,width=16.5cm,top=3cm,bottom=3cm]{geometry}

\numberwithin{equation}{section}

\usepackage[french,english]{babel}

\allowdisplaybreaks
\emergencystretch=\maxdimen
\usepackage{multicol}

\usepackage{xcolor}

\newtheorem{theorem}{Theorem}[section]

\newtheorem{proposition}[theorem]{Proposition}

\newtheorem{lemma}[theorem]{Lemma}
\newtheorem{remark}[theorem]{Remark}

\newtheorem*{definition*}{Definition}

\newcommand{\supp}{{\rm supp}}

\newcommand{\dist}{\mathop{\mathrm{dist}}\nolimits}

\newcommand{\ddc}{{\rm dd^c}}
\newcommand{\dc}{{\rm d^c}}

\newcommand{\dd}{{\rm d}}

\newcommand{\dbar}{\overline\partial}

\newcommand{\vep}{\varepsilon}

\newcommand{\cM}{\mathcal{M}}

\newcommand{\FS}{{\rm FS}}
\newcommand{\B}{\mathbb{B}}

\newcommand{\D}{\mathbb{D}}
\newcommand{\C}{\mathbb{C}}

\renewcommand\P{\mathbb{P}}

\newcommand{\lp}{\langle}
\newcommand{\rp}{\rangle}

\newcommand{\oL}{\mathcal{L}}

\newcommand{\oI}{\mathcal{I}}

\newcommand{\fh}{\mathfrak{h}}

\newcommand{\bP}{\mathbf{P}}

\newcommand{\bd}{\mathbf{bd}}



\usepackage{fancyhdr}
\title{Asymptotics of hole probability regarding open balls for random sections on compact Riemann surfaces}

\author{Hao Wu}
\address{School of Mathematics,  Nanjing University - Nanjing - China 210093}
\email{haowu@nju.edu.cn}

\date{}
\thanks{}

\begin{document}

\begin{abstract}
We obtain the asymptotic behavior of hole probability for random holomorphic sections on a compact Riemann surface with respect to the hole size.
\end{abstract}

\clearpage\maketitle
\thispagestyle{empty}


\medskip


\setcounter{tocdepth}{1}
 \normalfont

\section{Introduction}

Let \( (X,\omega_0) \) be a compact Riemann surface, and let \( \mathcal{L} \) be a positive holomorphic line bundle on \( X \) with \( \deg(\mathcal{L}) \geq 1 \). We fix a Hermitian metric \( \fh \) on \( \mathcal{L} \) such that the Chern curvature form \( c_1(\oL, \fh) \) is strictly positive. The normalized \((1,1)\)-form
\[
\omega := c_1(\oL, \fh)/\deg(\mathcal{L})
\]
is a smooth probability measure on \( X \), since \( \int_X \omega = 1 \).

For every positive integer \( n \), the \( n \)-th power \( \mathcal{L}^n := \mathcal{L}^{\otimes n} \) of the line bundle \( \mathcal{L} \) inherits a natural metric \( \fh_n \) induced by \( \fh \). Specifically, for any holomorphic section $s$ of $\oL$, $s^{\otimes n}$ is a holomorphic section of $\oL^n$, and we have
$$ \|s^{\otimes n}\|_{\fh_n} (x):= \|s\|_{\fh}^n (x)  \quad\text{for every}\quad x\in X. $$ 
Let \( (\cdot, \cdot)_n \) be  the  Hermitian inner product at each point $x$ corresponding to the Hermitian metric \( \fh_n \).

On the space \( H^0(X, \mathcal{L}^n) \) of global holomorphic sections of \( \mathcal{L}^n \), we define a global Hermitian inner product as follows:
\begin{equation*}
\langle s_1, s_2 \rangle_n := \int_X (s_1(x), s_2(x))_n \, \omega_0 (x) \quad \text{for} \quad s_1, s_2 \in H^0(X, \mathcal{L}^n).
\end{equation*}

The Riemann-Roch theorem says that
\[
\dim_{\mathbb{C}} H^0(X, \mathcal{L}^n) = n \cdot \deg(\mathcal{L}) - g + 1.
\]
The projectivized space \( \mathbb{P}H^0(X, \mathcal{L}^n) \) is well-defined.  We shall denote by \( V^{\FS}_n \) the Fubini-Study volume form on \( \mathbb{P}H^0(X, \mathcal{L}^n) \) induced by \( \langle \cdot, \cdot \rangle_n \).

\smallskip

The zero set of a section in $H^0(X,\oL^n)\setminus\{0\}$ doesn't change if we multiply the section by a non-zero constant in $\C$. Therefore, we can denote by $Z_s$ the zero set of a section $s$ in $H^0(X,\oL^n)\setminus\{0\}$ or of an element $s$ in $\P H^0(X,\oL^n)$. The points in $Z_s$ are counted with multiplicity. So $Z_s$ defines an effective divisor of degree $n\deg(\oL)$ that we still denote by $Z_s$. Let $[Z_s]$ be the sum of Dirac masses of the points in $Z_s$ and $$\llbracket Z_s \rrbracket:= n^{-1}\deg(\oL)^{-1} [Z_s]$$
the \textit{empirical measure} with respect to the section $s$. 
 
The study of the zeros of  sections in $H^0(X,\oL^n)\setminus\{0\}$ with respect to the standard complex Gaussian on $H^0(X,\oL^n)$ is equivalent to the study of the zeros of  elements of 
$\P H^0(X,\oL^n)$ with respect to the probability measure $V_n^\FS$, see e.g., \cite[Section 2]{shi-zel-gafa}.  In what follows, by a \textit{random section}, we mean a random element in $H^0(X,\oL^n)\setminus\{0\}$ with respect to the standard complex Gaussian or a random element in $\P H^0(X,\oL^n)$ with respect to  $V_n^\FS$.

A celebrated theorem by Shiffman and Zelditch~\cite{shi-zel-cmp} states that  the zeros of random sections  are equidistributed with respect to \( \omega \). More precisely, for any smooth test function \( \phi \) on \( X \), one has
\[
\lim_{n \to \infty} \int_{\mathbb{P}H^0(X, \mathcal{L}^n)} \lp \llbracket Z_s \rrbracket, \phi \rp \, V^{\FS}_n(s) = \int_X \phi \, \omega.
\]

Our goal is to study the hole probabilities of this distribution. Namely,  for an open subset \( D \) of \( X \) with \( \overline{D} \neq X \), we define, for each large \( n \), the \textit{hole event}
\[
H_{n,D} := \big\{ [s] \in \mathbb{P}H^0(X, \mathcal{L}^n) \mid Z_s \cap D = \varnothing \big\},
\]
which consists all holomorphic sections of \( \mathcal{L}^n \)   non-vanishing on \( D \).  Define the \textit{hole probability}
$$ \bP_n (H_{n,D}):= \int_{H_{n,D}}    V^{\FS}_n.       $$
This quantity has been studied by many researchers in the past two decades, see e.g., \cite{dlm-jussieu,shi-zel-zre-ind,zhu-anapde, zre-michigan}. Recently, the author and  Xie  \cite{WX-hole} derived the optimal convergence speed of the hole probability as $n\to \infty$. It is worth mentioning that the author \cite{DGW-hole-event} also  proved that the zeros of random sections in $H_{n,D}$ are equidistributed, together with  Dinh and Ghosh.

\medskip

In this article, we only consider the case $D=\B(x,r)$ and  focus on the asymptotic behavior of $\bP_n(H_{n, \mathbb{B}(x, r)})$ as the radius  $r\to 0$, measured with respect to the K\"ahler metric \(\omega_0\).
The following is our main result of this article.

\begin{theorem} \label{main theorem}
For any $x\in X$, as $r\to 0$, there exist a constant \( C_x > 0 \) such that
\begin{equation}
\label{thm 2 estimate}
 \lim_{n \to \infty} \frac{1}{n^2 \deg(\mathcal{L})^2} \log \bP_n(H_{n, \mathbb{B}(x, r)}) =-C_x e^2 \pi^2 r^4+O(r^5).
\end{equation}
\end{theorem}

The value of $C_x$ is determined by the formal equation: $$\omega(x)=2\sqrt {C_x} \, \omega_0(x).$$ In particular, when $\omega_0=\omega$, $C_x=1/4$ for all $x\in X$.

\medskip

In \cite{WX-hole}, the author and Xie  established that (see also \cite{DGW-hole-event})
  $$ \Big|\frac{1}{n^2 \deg(\oL)^2} \log \bP_n(H_{n,D}) + \min \oI_{\omega,D} \Big|  =O\Big( \frac{\log n}{n}\Big)  \quad\text{as}\quad n\to \infty,$$
where $\oI_{\omega,D}$ is a functional defined on the space of all probability measures on  $X\setminus D$, which will be introduced in Section \ref{sec-pre}. Thus, to prove Theorem \ref{main theorem}, it is enough to show 

\begin{theorem}\label{thm-main-hole}
There exist positive constants  $c_x,C_x$ independent of $r$, such that for all $r>0$,
$$ | \min \oI_{\omega,\B(x,r)} -C_x  e^2   \pi^2 r^4 | \leq c_x r^5.      $$
\end{theorem}

\begin{remark}\rm
By a compactness argument, one can take the $c_x$ in Theorem \ref{thm-main-hole} independent of $x$, and hence, the error term $O(r^5)$ in Theorem \ref{main theorem} is also independent of $x$.
\end{remark}

Previously, such hole probabilities with a parameter \( r \) on the hole size were not known for zeros of random holomorphic sections, except in some very special cases where some direct calculations are possible \cite{zhu-anapde, zre-michigan}. 

\smallskip

In fact, instead of random sections,
the estimate~\eqref{thm 2 estimate} is inspired by several works of \textit{random polynomials}, which we now mention. 
In the setting of Gaussian entire functions \[ F(z) := \sum_{n=1}^\infty \frac{\zeta_n}{\sqrt{n!}} z^n \] with independent and identically distributed (i.i.d.)\ standard complex Gaussian coefficients \( \zeta_n \), Sodin and Tsirelson~\cite{sod-tsi-isrj} showed that the exponential decay speed of hole probability with respect to a disc of radius \( r \) is \( \exp(-c r^4) \) as \( r \to \infty \). The optimal constant $c$ was later obtained by Nishry  \cite{nis-imrn}. Recently, Buckley, Nishry, Peled, and Sodin~\cite{buc-nis-ron-sod-ptrf} studied the hole probabilities for zeros of hyperbolic Gaussian Taylor series with finite radii of convergence. See also~\cite{  gho-nis-con, gho-nis-cpam, cho-zei-imrn, kri-jsp, kur-ska-math-stud, nis-jdm, Nishry-Wennman}.

\medskip

The paper is organized as follows. In Section \ref{sec-pre}, we introduce some useful tools from complex analysis and potential theory. In Section \ref{sec-flat}, we compute the exact value of the hole probability under the flat assumption, which gives the leading term of \eqref{thm 2 estimate} in Theorem \ref{main theorem}. The error term of \eqref{thm 2 estimate} will be proved in Sections \ref{sec-bundle} and \ref{sec-dis}. Before that, we will establish  a crucial localization of the problem in Section \ref{sec-loc}.

\medskip
\section{Preliminaries} \label{sec-pre}

In this section, we will introduce the key functional $\oI_{\omega,D}$, and some useful notion in complex analysis, as well as potential theory.

\medskip

Denote by $\cM(X)$ the set of all probability measures on $X$. It carries the following natural \textit{weak topology}: a sequence of probability measures $\mu_n$ converges to $\mu$ weakly, if for any smooth function $\phi$, one has $\lim_{n\to\infty}\int \phi \,\dd \mu_n=\int \phi \,\dd \mu$. Similarly, for any closed subset $K\subset X$, we can define the restriction   $\cM(K)$, which is compact and convex under the weak topology.

\smallskip

A function \(\phi\) on \(X\) with values in \(\mathbb{R} \cup \{-\infty\}\) is called \textit{quasi-subharmonic} if, locally, it can be written as the difference of a subharmonic function and a smooth function. If \(\phi\) is quasi-subharmonic, then there exists a constant \(c \geq 0\) such that \(\ddc \phi \geq -c  \omega\) in the sense of currents ($\dc:=\frac{i}{2\pi}(\dbar-\partial)$ and $\ddc=\frac{i}{\pi}\partial\dbar$). When \(c = 1\), \(\phi\) is called an \textit{\(\omega\)-subharmonic function}, and \(\ddc \phi + \omega\) is a probability measure on \(X\) by Stokes' formula.

For any probability measure \(\mu\) on \(X\), we can write \(\mu = \omega + \ddc U_\mu\), where \(U_\mu\) is the unique quasi-subharmonic function such that \(\max U_\mu = 0\). We call \(U_\mu\) the \(\omega\)-\textit{potential} of \(\mu\).
There is an alternative way to normalize the potential \(U^\star_\mu\) by requiring that \(\int_X U^\star_\mu \, \omega = 0\). We call \(U^\star_\mu\) the  \(\omega^*\)-\textit{potential}   of \(\mu\). By definition,
\begin{equation}\label{diff-potential} U_\mu=U^\star_\mu-\max_X U^\star_\mu.    
\end{equation}

For any \textbf{simply connected} open set $D\subset X$ with \textbf{smooth boundary}, we define
\begin{equation}
\label{functional I_r}
\mathcal{I}_{\omega, D}(\mu) := -\int_X U_\mu \, \omega - \int_X U_\mu \, \dd\mu , \quad \forall \, \mu \in \mathcal{M}(X \setminus D).
\end{equation}
When \(D\) is non-empty, \(\mathcal{I}_{\omega, D}\) is strictly positive.
Recall the following result concerning the functional $\oI_{\omega,D}$ on $\cM(X\setminus D)$ from \cite{DGW-hole-event}.

\begin{lemma}\label{l:convex}
As a functional on   $\cM(X\setminus D)$ endowed with the weak topology, 
$\oI_{\omega,D}$ is lower semi-continuous and strictly convex on the set $\{\oI_{\omega,D}\not=+\infty\}$. It admits a unique minimizer $\nu$ on $\cM(X\setminus D)$ satisfying 
$$ U_\nu=    \sup_\phi  \big\{  \phi \text{ is } \omega\text{-subharmonic}: \; \phi\leq 0 \text{ on } X,\; \phi \leq U_\nu \text{ on } \overline D  \big\}.$$ 
\end{lemma}

For convenience, we shall abbreviate \(\mathcal{I}_{\omega,r} := \mathcal{I}_{\omega, \B(x,r)}\), since we only consider $D=\B(x,r)$ and  \(x\)  is fixed throughout this article.  Write $\oI_\omega:=\oI_{\omega,\varnothing}$.  By Lemma \ref{l:convex} above, $\mathcal{I}_{\omega,r}$ admits a \textbf{unique minimizer} \(\nu_{\omega,r} \in \mathcal{M}(X \setminus \B(x,r))\). Moreover, $U_{\nu_{\omega,r}}$ is continuous due to the smoothness of $\partial D$.

We have the follow monotone property related to $\oI_{\omega}$ and $\omega$-potential, which will be used very frequently later.

\begin{lemma}\label{lem-compare-oI}
If $\sigma,\eta$ are two probability measures on $X$ such that $ U_{\sigma}\leq U_\eta $ on $X$. Then 
$$   \oI_\omega (\eta)\leq \oI_\omega(\sigma).    $$
\end{lemma}

\begin{proof}
 Using Stoke's formula several times, we have 
\begin{align*}
&\int U_\sigma \, \dd \sigma- \int U_\eta \,\dd\eta=\int U_\sigma \, \dd \sigma -\int U_\sigma \, \dd \eta +\int U_\sigma \, \dd \eta- \int U_\eta \,\dd\eta \\
&=\int U_\sigma \,\ddc( U_\sigma-U_\eta)+ \int (U_\sigma -U_\eta)\,\dd \eta\leq \int U_\sigma \,\ddc( U_\sigma-U_\eta)+0\\
&= \int \ddc U_\sigma \, ( U_\sigma -U_\eta)   =\int ( U_\sigma -U_\eta) \,\dd(\sigma-\omega)  \\
&=  \int ( U_\sigma -U_\eta) \,\dd\sigma -\int ( U_\sigma -U_\eta) \,\omega \leq 0  -\int ( U_\sigma -U_\eta) \,\omega.
\end{align*}
This gives the desired inequality by definition \eqref{functional I_r}.
\end{proof}

For a \textbf{negative} continuous function $u$ on $\partial D$, define the upper envelop
$$  \widehat U:=    \sup_\phi  \big\{  \phi \text{ is } \omega\text{-subharmonic}: \; \phi\leq 0 \text{ on } X,\; \phi \leq u \text{ on }  \partial D  \big\}.$$    

The proof of next lemma should be standard, but we cannot find an exact the same statement in literature. So we provide the details for convenience.

\begin{lemma}\label{lem-upenvop}
The function $\widehat U$ is a continuous $\omega$-subharmonic functions on $X$ satisfying
$$ \ddc \widehat U=-\omega   \,\text{ on }\, \{\widehat U\neq 0 \}\setminus \partial D   \quad\text{and}\quad  \widehat U=u  \,\text{ on }\, \partial D.         $$
\end{lemma}

\begin{proof}
\textbf{Step 1:} $\widehat U$ is $\omega$-subharmonic.
\smallskip

Clearly, $\ddc \widehat U\geq -\omega$. To prove $\widehat U$ is $\omega$-subharmonic, we need to show that $\widehat U$ is upper semi-continuous. Let $\widehat U^*$   be the upper semi-continuous regularization of $\widehat U$, which is $\omega$-subharmonic. In the following, we will prove $\widehat U=\widehat U^*$.

Let $V$ be the continuous function on $X$ satisfying
$$V=u \,\text{ on }\,  \partial D,  \quad \ddc V=-\omega  \,\text{ on }\, X\setminus \partial D.       $$
This $V$ is the unique solution of Dirichlet problem on the domains $D$ and $X\setminus \overline D$. The continuity of $V$ is guaranteed by the continuity of $u$. 

 For any $\omega$-subharmonic function $\phi$ such that $\phi\leq 0$ on $X$ and $\phi\leq u$ on $\partial D$, we have
$\phi \leq \widehat U\leq \widehat U^*$. Applying maximal modulus principle to the function $\phi - V$, which is subharmonic on both $D$ and $X\setminus \overline D$, we get 
$\phi \leq V$   on $X$.  
It follows that 
$ \widehat U \leq V$   on $X$, and hence
$$\widehat U^* \leq V  \,\text{ on }\, X   $$
because $V$ is continuous. In particular, $\widehat U^* \leq u$ on $\partial D$. Thus, $\widehat U^*$ itself is an $\omega$-subharmonic function satisfying $\widehat U^*\leq 0$ on $X$ and $\widehat U^*\leq u$ on $\partial D$. This gives
$$\widehat U^* \leq \widehat U   \,\text{ on }\, X .   $$
So we conclude that $\widehat U=\widehat U^*$, finishing the proof of Step 1.

\medskip
\noindent \textbf{Setp 2:} $ \ddc \widehat U=-\omega$ on $\{\widehat U\neq 0 \}\setminus \partial D$.
\smallskip

Take a point $y\in \{\widehat U\neq 0\}\setminus\partial D$. Since $\widehat U$ is upper semi-continuous by Step 1, we can take two small open balls $B_1,B_2$ and an $\vep>0$ such that $$x\in B_1\Subset B_2\Subset \{\widehat U\neq 0\}\setminus\partial D \quad\text{and}\quad  \widehat U\leq -2\vep \,\text{ on }\,B_2$$ 

Suppose for contradiction, $\ddc\widehat U\neq -\omega$ near $y$, which means $\ddc \widehat U+\omega \neq 0$ on any open neighborhood of $y$. 
After shrinking $B_1,B_2$, we may fix a smooth function $\varphi$   on $B_2$ such that
$$|\varphi|\leq \vep\,\text{ on }\,B_2 \quad\text{and}\quad \ddc \varphi =\omega.$$ Then $\widehat U + \varphi$ is subharmonic on $B_2$ and not harmonic on $B_1$.
By \cite[Prop. 9.1]{bedford-1982}, we can find a subharmonic function $\psi$ on $B_2$ such that 
$$\ddc \psi =0   \,\text{ on }\, B_1\quad\text{and}\quad   \psi=    \widehat U +\varphi  \,\text{ on }\,  B_2\setminus B_1.$$
Moreover,   maximal modulus principle gives $\sup_{B_2} \psi=\sup_{B_2\setminus B_1}\psi$, which implies $\psi\leq \sup_{B_2\setminus B_1} (\widehat U +\varphi)\leq -\vep$  on $B_1$.  Applying maximal modulus principle to   $\widehat U+\varphi-\psi$, we see that 
$\psi> \widehat U +\varphi \,\text{ on }\, B_1.$
Therefore, the function $\Psi$ defined as 
$$\Psi:=\psi-\varphi    \,\text{ on }\, B_1\quad\text{and}\quad   \Psi:=    \widehat U    \,\text{ on }\,  X\setminus B_1 $$
is an $\omega$-subharmonic function satisfying $\Psi\leq 0$ on $X$ and $\Psi\leq u$ on $\partial D$. This contradicts to the fact that $\widehat U$ is the maximal one among all these kind of functions.

\medskip
\noindent \textbf{Setp 3:} $\widehat U=u$ on $\partial D$.
\smallskip
The proof is similar as Step 2. Suppose $w$ is a point in $\partial D$ such that $\widehat U(w)<u(w)$. Then we can take an $\vep>0$ and two small open balls $B_1,B_2$ such that 
$$w\in B_1\Subset B_2 \quad\text{and}\quad  \widehat U\leq u-2\vep\,\text{ on }\, \partial D \cap B_2,\quad  \widehat U\leq -2\vep \,\text{ on }\, B_2.$$
By the same construction as in Step 2, we can find an $\omega$-subharmonic function $\Psi$  satisfying $\Psi\leq 0$ on $X$ and $\Psi\leq u$ on $\partial D$, which gives the contradiction.

\medskip
\noindent \textbf{Setp 4:} $\widehat U$ is continuous.
\smallskip
 
We first prove    the continuity on $\overline D$. Denote $D_1:=\{\widehat U\neq 0\}\cap D$.  Note that $\widehat U$ is the solution of the following Dirichlet problem
$$ \ddc \widehat U =-\omega \,\text{ on }\, D_1, \quad  \widehat U=u \,\text{ on }\, \partial D, \quad \widehat U=0 \,\text{ on }\, \partial D_1\setminus \partial D.    $$
The boundary data is continuous. So $\widehat U$ is continuous on $\overline D_1$ and hence it is continuous on $\overline D$. The proof of continuity on $X\setminus D$ is similar.
\end{proof}

\medskip
\section{Flat case} \label{sec-flat}

In this section, we will  consider   a simple situation, assuming that near $x$, the two metrics $\omega$ and $\omega_0$ are both flat, i.e., there exist an $r_0>0$ and a local coordinate $z$ such that, on $\B(x,r_0)$, 
\begin{equation}\label{eqn-flat}
\omega=\alpha \, i\dd z\wedge \dd \overline z, \quad   \omega_0= \beta \, i\dd z\wedge \dd \overline z
\end{equation}
for some $\alpha,\beta>0$. We may further assume  that $z=0$ at $x$ and $\beta=1/2$, in which case,  $z$ is an isometry from $\B(x,r_0)$ to $\D(0,r_0)$.

\begin{proposition}\label{prop-r0}
Under condition \eqref{eqn-flat}, we have $$\min\oI_{\omega,r}=\alpha^2 e^2 \pi ^2 r^4 \quad\text{for $r>0$ small enough}.$$
\end{proposition}

\begin{proof}
Recall from \cite{DGW-hole-event} that $\nu_{\omega,r} := \omega|_{S_{\omega,r}} + \nu_{\bd r}$, where $S_{\omega,r} := \{ U_{\nu_{\omega,r}} = 0 \} \setminus \overline{\B(x,r)}$ and $\nu_{\bd r}$ is a non-vanishing positive measure on the boundary of \(\B(x,r)\).  Moreover, $S_{\omega,r}\cap  \overline{\B(x,r)}=\varnothing$. As $r\to 0$, the set $\{ U_{\nu_{\omega,r}} \neq 0 \}$ will shrinking to the point $x$. Thus, we may take small enough $r$ to assume that 
\begin{equation*}
\{ U_{\nu_{\omega,r}} \neq 0 \}\subset \B(x,r_0).
\end{equation*}
In which case, the metric $\omega$ and $\omega_0$ are flat on $\{ U_{\nu_{\omega,r}} \neq 0 \}$. By the uniqueness of $\nu_{\omega,r}$ and symmetry, $U_{\nu_{\omega,r}}$ is radial under the coordinate $z$ on $\B(x,r_0)$, i.e., it is a function of $|z|:=\dist(z,x)$.  In particular, $U_{\nu_{\omega,r}}$ is constant on $\partial\B(x,r)$, which we assume to be some negative constant $\gamma$.  By Lemma \ref{l:convex},
$$ U_{\nu_{\omega,r}}=    \sup_\phi  \big\{  \phi \text{ is } \omega\text{-subharmonic}: \; \phi\leq 0 \text{ on } X,\; \phi \leq U_{\nu_{\omega,r}} \text{ on } \overline \B(x,r)  \big\}.$$
This implies that $ U_{\nu_{\omega,r}}$ is uniquely determined by $\gamma$.

On the other hand, note that $\ddc U_{\nu_{\omega,r}}=-\omega$ because $\supp(\nu_{\omega,r})\subset X\setminus\B(x,r)$. Combining with Lemmas \ref{lem-compare-oI} and \ref{lem-upenvop},  we see that the maximal choice of $\gamma$ gives $U_{\nu_{\omega,r}}$, in which case, $U_{\nu_{\omega,r}}(x)=0$ by symmetry. By a direct computation, we obtain $\gamma=-\alpha\pi r^2$ and  (letting $|z|:=\dist(z,x)$)
\begin{equation}  \label{solution-flat}
U_{\nu_{\omega,r}} (z)=
\begin{cases}  - \alpha\pi |z|^2          &\quad\text{for}\quad  z\in \B(x,r)\\
 \alpha\pi \Big(2er^2\log{ |z|\over r}- |z|^2 \Big)   &\quad\text{for}\quad  z\in \B(x,\sqrt e \, r)\setminus \B(x,r)\\
 0 &\quad\text{for}\quad  z\in X\setminus \B(x,\sqrt e \, r).
\end{cases}
\end{equation}

Now we can compute the value of $\oI_{\omega,r}(\nu_{\omega,r})$ for the flat case.  Since $U_{\nu_{\omega,r}} =0$ on $X\setminus\B(x,\sqrt e \, r)$, we can work on $\B(x,\sqrt e \,r)$ using the coordinate $z$. With the help of polar coordinates, we have
\begin{align*}
&-\int_X U_{\nu_{\omega,r}}\, \omega=-\int_{\{ U_{\nu_{\omega,r}} \neq 0 \}} U_{\nu_{\omega,r}}\, \omega \\
&= \int_{\D(0,r)}  \alpha\pi |z|^2 \alpha \, i\dd z\wedge \dd \overline z +\int_ {\D(0,\sqrt e \,r)\setminus \D(0,r)} \alpha\pi \Big(|z|^2-2er^2\log{|z|\over r}\Big) \alpha \, i\dd z\wedge \dd \overline z\\
&=\alpha^2 \pi \int_0^{2\pi}\int_0^r t^2 2t \,\dd t \dd \theta +\alpha^2 \pi \int_0^{2\pi}\int_r^{\sqrt e \, r} \Big(t^2-2er^2\log{t\over r}\Big) 2t \,\dd t \dd \theta \\
&= \alpha^2 \pi ^2 r^4+ ( e^2-2e-1 )\alpha^2 \pi ^2 r^4=  (e^2-2e)\alpha^2 \pi ^2 r^4.
\end{align*}
For the other term $-\int_X U_{\nu_{\omega,r}}\, \dd\nu_{\omega,r}$, note that $U_{\nu_{\omega,r}}=0$ on $S_{\omega, r}$, and   the mass of $\nu_{\bd r}$ is 
$ 1-  \omega(S_{\omega, r})$, which equals  the area of $\D(0,\sqrt e \,r)$ under the metric $\omega$. Indeed, $\omega(S_{\omega, r})+\nu_{\bd r}(X)=1$. Using that $U_{\nu_{\omega,r}}=-\alpha \pi r^2$  on $\supp(\nu_{\bd r})=\partial\B(x,r)$, we have
$$ -\int_X U_{\nu_{\omega,r}}\, \dd\nu_{\omega,r}
=-\int_{X} U_{\nu_{\omega,r}}\, \nu_{\bd r}= \alpha \pi r^2\cdot 2\alpha \pi e r^2 =2e\alpha^2 \pi^2 r^4.$$
The proposition follows.
\end{proof}

\begin{remark}\rm
Under condition \eqref{eqn-flat}, the minimum value of $\oI_{\omega,r}$ is calculable, because of the following key fact:
\begin{equation}\label{symcase}
\omega \,\text{ is flat on }\,  \{ U_{\nu_{\omega,r}} \neq 0 \},
\end{equation}
so that the symmetric argument can be applied.
\end{remark}

\begin{remark}\rm
In fact, the function $U_{\nu_{\omega,r}}$ defined in \eqref{solution-flat} is exactly the upper envelop
$$  \widehat U:=    \sup_\phi  \big\{  \phi \text{ is } \omega\text{-subharmonic}: \; \phi\leq 0 \text{ on } X,\; \phi \leq u \text{ on }  \partial \B(x,r)  \big\},$$
where $u$ is the constant function $-\alpha\pi r^2$.
\end{remark}

\medskip
\section{Localization} \label{sec-loc}

We want to weaken the condition \eqref{eqn-flat} as $r_0$ is to large comparing with $r$. In this section, we only assume
\begin{equation}\label{eqn-flat-2}
\omega=\alpha \, i\dd z\wedge \dd \overline z   \,\text{ on }\, \B(x,2r), \quad   \omega_0= 1/2 \, i\dd z\wedge \dd \overline z \,\text{ on }\, \B(x,r_0)  
\end{equation}
for some $\alpha>0$, where  $z$ is a local coordinate on $\B(x,r_0)$ such that $z=0$ at $x$. 

\begin{proposition}\label{prop-2r}
Under condition $\eqref{eqn-flat-2}$, we have $$\min\oI_{\omega,r} =\alpha^2 e^2  \pi ^2 r^4.$$
\end{proposition}

Comparing with condition \eqref{eqn-flat}, the difficult of this case is that, we only know $\omega$ is flat on a quite small neighborhood of $x$, which does not implies \eqref{symcase} directly. So one cannot use the same argument as in Proposition \ref{prop-r0} to conclude the proof.  

\medskip
Let $\sigma_r$ be the minimizer of $\oI_{\omega,r}$ under  condition \eqref{eqn-flat}, in other words, its $\omega$-potential $U_{\sigma_r}$ is defined in \eqref{solution-flat} as follows (letting $|z|:=\dist(z,x)$):
\begin{equation}  \label{defn-Phir}
U_{\sigma_r} (z)=
\begin{cases}  - \alpha\pi |z|^2          &\quad\text{for}\quad  z\in \B(x,r)\\
 \alpha\pi \Big(2er^2\log{ |z|\over r}- |z|^2 \Big)   &\quad\text{for}\quad  z\in \B(x,\sqrt e \, r)\setminus \B(x,r)\\
 0 &\quad\text{for}\quad  z\in X\setminus \B(x,\sqrt e \, r).
\end{cases}
\end{equation}

In the proof of Proposition \ref{prop-r0}, we have computed that $$\oI_{\omega,r}(\sigma_r)=\alpha^2 e^2 \pi ^2 r^4,$$ and hence  $\min\oI_{\omega,r} \leq \alpha^2 e^2 \pi ^2 r^4$.
It  remains to  prove this is also the lower bound. Equivalently, we need to show that under condition \eqref{eqn-flat-2}, $\sigma_r$ is the minimizer of $\oI_{\omega,r}$ as well.

\smallskip

Consider the following proper subset of $\cM(X\setminus \B(x,r))$:
$$\Omega_r:=\big\{\mu\in\cM(X\setminus \B(x,r)):\, U_\mu=0 \text{ on } X\setminus \B(x,2r) \big\}.     $$

\begin{lemma}\label{lem-miniOmega}
Under condition $\eqref{eqn-flat-2}$, the minimum of $\oI_{\omega,r}$ over $\Omega_r$ appears at $\sigma_r$. In particular, $\oI_{\omega,r}(\mu)\geq \alpha^2 e^2 \pi ^2 r^4$ for all $\mu\in\Omega_r$.
\end{lemma}

If we know the minimum of $\oI_{\omega,r}$ over $\Omega_r$ is unique, we can just apply the symmetry argument to conclude. However,
it is not clear that $\Omega_r$ is closed and convex under the weak topology of probability measures. One cannot  use the convexity of $\oI_{\omega,D}$ in Lemma \ref{l:convex} to get the uniqueness of the minimum over $\Omega_r$.

\begin{proof}[Proof of Lemma \ref{lem-miniOmega}]
For every $\mu\in \Omega_r$, we define the ``symmetric function" on $X$:
$$ \widetilde U_\mu (z) :={1\over 2\pi}\int_0^{2\pi} U_\mu(e^{i\theta} z)\,\dd \theta  \,\text{ for }\, z\in\B(x,2r), \quad   \widetilde U(z)=0 \,\text{ for }\, z\notin\B(x,2r).$$
Here, $e^{i\theta} z$ is well-defined  due to the flat assumption on $\omega_0$. It is not hard to see that $\widetilde U_\mu$ is still $\omega$-subharmonic and satisfying
$$\max \widetilde U_\mu=0 ,\quad  \ddc \widetilde U_\mu=-\omega \,\text{ on }\, \B(x,r).$$ 
Thus, it is the $\omega$-potential of the probability measure  $\widetilde\mu:=\ddc \widetilde U_\mu+\omega$, whose support is outside $\B(x,r)$. 

We want to bound $\oI_{\omega,r}(\widetilde\mu)$.  Define the sequence of functions $V_n$ as follows: $V_n=0$ on $X\setminus \B(x,2r)$, and 
$$ V_n(z) ={1\over n} \sum_{k=1}^n  U_\mu (e^{2ik\pi/n } z )        \,\text{ for }\, z\in\B(x,2r).$$  
Clearly, $V_n$ is $\omega$-subharmonic and 
$$\max  V_n=0 ,\quad  \ddc  V_n=-\omega \,\text{ on }\, \B(x,r).$$
So, $V_n$ is the $\omega$-potential of the probability measure $\mu_n:=\ddc V_n+\omega$, whose support is outside $\B(x,r)$ as well.  Moreover, we have 
$$ \mu_n= {1\over n} \sum_{k=1}^n \mu_{n,k}, $$
where $\mu_{n,k}:=\ddc U_\mu (e^{2ik\pi/n } z )  +\omega$. By definition,
\begin{align*}
\oI_{\omega,r}(\mu_{n,k}) &= -\int_X U_\mu (e^{2ik\pi/n } z ) \, \omega (z) - \int_X U_\mu (e^{2ik\pi/n } z ) \, \dd\mu_{n,k}(z) \\
&=  -\int_{\B(x,2r)} U_\mu (e^{2ik\pi/n } z ) \, \omega (z) - \int_{\B(x,2r)} U_\mu (e^{2ik\pi/n } z ) \, \dd\mu_{n,k}(z) \\
&= -\int_{\B(x,2r)} U_\mu    \, \omega   - \int_{\B(x,2r)} U_\mu   \, \dd\mu   =\oI_{\omega,r}(\mu).
\end{align*}
Recalling the convexity of $\oI_{\omega,r}$ from Lemma \ref{l:convex}, we get 
$$ \oI_{\omega,r}(\mu_n)  \leq {1\over n}\sum_{k=1}^n \oI_{\omega,r}(\mu_{n,k})= {1\over n}\sum_{k=1}^n \oI_{\omega,r}(\mu )  = \oI_{\omega,r}(\mu).  $$

On the other hand, note that $V_n$ converge to $\widetilde U_\mu$, which implies $\mu_n$ converge to $\widetilde\mu$ weakly. Using Lemma \ref{l:convex} again, we have 
$$ \liminf_{n\to\infty} \oI_{\omega,r}(\mu_n) \geq \oI_{\omega,r}(\widetilde \mu).     $$
Therefore, we conclude that 
$$ \oI_{\omega,r}(\widetilde \mu)\leq   \oI_{\omega,r}(\mu).     $$

Summing up,  we see that the minimum of $\oI_{\omega, r}$ over $\Omega_r$ attends in the following subset
$$\widetilde\Omega_r:=\big\{\mu\in\cM(X\setminus \B(x,r)):\, U_\mu=0 \text{ on } X\setminus \B(x,2r) ,\; U_\mu \,\text{ is radial on } \,  \B(x,2r) \big\}.     $$
In the proof of Proposition \ref{prop-r0}, we know that the minimum of $\oI_{\omega,r}$ over $\widetilde\Omega_r$ appears at $\sigma_r$, proving the lemma. 
\end{proof}

\smallskip

\begin{proof}[Proof of Proposition \ref{prop-2r}]
Suppose for contradiction,  $\min\oI_{\omega,r} =\oI_{\omega,r}(\nu_{\omega,r})<e^2 \alpha^2 \pi ^2 r^4$.  This $\nu_{\omega, r}$ may not lie in $\Omega_r$. We will start with this measure, constructing another probability measure $\eta\in\Omega_r$, whose functional value $\oI_{\omega, r}(\eta)$ is closed enough to $\oI_{\omega,r}(\nu_{\omega, r})$. Lastly, we apply Lemma \ref{lem-miniOmega} to get a contradiction.

\smallskip

Fix an $\vep>0$ small.
Consider the probability measure
$$\mu:=(1-\vep)\sigma_r +\vep \nu_{\omega,r}$$
and  its $\omega$-potential $U_\mu$ and  $\omega^*$-potential $U_\mu^*$.  Observe that $\mu\in\cM(X\setminus\B(x,r))$ and
$U_\mu^*=(1-\vep)U^*_{\sigma_r} +\vep U^*_{\nu_{\omega,r}}$. Using the continuity of $U^*_{\sigma_r}$ and $U^*_{\nu_{\omega,r}}$, we see that 
$$\min_{\partial \B(x,r)} U^*_{\mu}\geq (1-\vep) \min_{\partial \B(x,r)}U^*_{\sigma_r}+\vep\min_{\partial \B(x,r)}U^*_{\nu_{\omega,r}}$$
and 
$$ \max_X U^*_{\mu}\leq (1-\vep)\max_X   U^*_{\sigma_r}+\vep\max_X U^*_{\nu_{\omega,r}}. $$
Recall that $U_{\sigma_r}=-\alpha\pi r^2$ on $\partial\B(x,r)$.  By \eqref{diff-potential},
\begin{align*}
\min_{\partial \B(x,r)} U_{\mu} & \geq (1-\vep) \min_{\partial \B(x,r)}U^*_{\sigma_r}+\vep\min_{\partial \B(x,r)}U^*_{\nu_{\omega,r}} - (1-\vep)\max_X   U^*_{\sigma_r}-\vep\max_X U^*_{\nu_{\omega,r}} \\
& =(1-\vep)(\min_{\partial \B(x,r)}U^*_{\sigma_r}-\max_X   U^*_{\sigma_r})+\vep( \min_{\partial \B(x,r)}U^*_{\nu_{\omega,r}}-\max_X U^*_{\nu_{\omega,r}})\\
&=(1-\vep)\min_{\partial \B(x,r)}U_{\sigma_r}+\vep\min_{\partial \B(x,r)}U_{\nu_{\omega,r}}  >  -\alpha\pi r^2-A\vep,
\end{align*}
where $A:= -\min_{\partial \B(x,r)}U_{\nu_{\omega,r}}>0$.

\smallskip

Now consider the following two upper envelops:
$$ \Psi_r:=    \sup_\phi  \big\{  \phi \text{ is } \omega\text{-subharmonic}: \; \phi\leq 0 \text{ on } X,\; \phi \leq U_\mu \text{ on } \partial   \B(x,r)  \big\}$$ 
and
$$ \Psi_A:=    \sup_\phi  \big\{  \phi \text{ is } \omega\text{-subharmonic}: \; \phi\leq 0 \text{ on } X,\; \phi \leq  U_{\sigma_r}-A\vep \text{ on } \partial  \B(x,r)  \big\}.          $$
By Lemma \ref{lem-upenvop}, they are continuous $\omega$-subharmonic function and 
\begin{equation}\label{psirpsiA}
\Psi_r\geq U_\mu,\quad \Psi_r\geq \Psi_A.
\end{equation}
Moreover,  since $U_\mu,U_{\sigma_r}-A\vep$ themselves are  $\omega$-subharmonic functions, by  maximal modulus principle, we have on $\overline \B(x,r)$,
   $$\Psi_r=U_\mu \quad\text{and}\quad \Psi_A=U_{\sigma_r}-A\vep .$$ 
   
For $\vep$ small enough,  a direct computation (recalling condition \eqref{eqn-flat-2}) shows that $\Psi_A=0$ outside a small neighborhood of $\overline \B(x,\sqrt e\, r)$. In which case, $\Psi_r=0$ on $X\setminus \B(x,2r)$ by \eqref{psirpsiA}, and thus, $\Psi_r$ is the $\omega$-potential of the probability measure $\eta:=\ddc \Psi_r+\omega$. Lemma \ref{lem-compare-oI} gives
$$  \oI_{\omega,r}( \eta )\leq  \oI_{\omega,r}( \mu ). $$
Observe that $\eta\in \Omega_r$.

On the other hand,   Lemma \ref{l:convex}  states  that the functional $\oI_{\omega,r}$ is strictly convex on the set $\{\oI_{\omega,r}\neq +\infty\}$, yielding
$$\oI_{\omega,r}( \mu )\leq (1-\vep)\oI_{\omega,r} (\sigma_r)+\vep\oI_{\omega,r}(\nu_{\omega,r}) <\alpha^2e^2\pi^2r^4. $$

  Summing up, we have constructed a probability measure $\eta$ in $\Omega_r$ such that $\oI_{\omega,r}(\eta)<\alpha^2e^2\pi^2r^4$. This contradicts to Lemma \ref{lem-miniOmega}.
\end{proof}

\medskip
\section{Perturb  line bundle metric} \label{sec-bundle}

In this section, we will relax the condition \eqref{eqn-flat-2} further, only assuming $\omega_0$ to be flat. In other words, there exist an $r_0>0$ and a local coordinate $z$ such that, on $\B(x,r_0)$,
\begin{equation}\label{eqn-flat-3}
 z=0 \,\text{ at }\, x \quad\text{and}\quad  \omega_0=1/2 \, i\dd z\wedge \dd \overline z.
\end{equation}

Since $\omega$ is smooth, we have near $x$, 
$$\omega(z)= (1+O(|z|) \alpha \,i \dd z\wedge \dd \overline z      $$
for some $\alpha>0$.
So, there exists a constant $\rho>0$ such that 
\begin{equation}\label{ineq-omega-z}
(1-\rho |z|)\alpha   \,i \dd z\wedge \dd \overline z   \leq \omega \leq (1+\rho |z|) \alpha \,i \dd z\wedge \dd \overline z \quad\text{on}\quad \B(x,r_0).
\end{equation}

\begin{lemma}\label{lem-lower bound}
Under condition \eqref{eqn-flat-3}, we have $$\min\oI_{\omega,r} \leq \alpha^2e^2 \pi ^2 r^4 +O(r^5).$$
\end{lemma}

\begin{proof}
We put $\vep:=2\rho r$ to simplify notation. Immediately from \eqref{ineq-omega-z}, we get
\begin{equation}\label{ineq-omega}
(1-\vep)\alpha   \,i \dd z\wedge \dd \overline z   \leq \omega \leq (1+\vep) \alpha \,i \dd z\wedge \dd \overline z  \quad\text{on}\quad \B(x,2r).
\end{equation}

On $\B(x,r_0)$, we define two local K\"ahler form   
$$\omega_1:=(1-\vep)\alpha \,i \dd z\wedge \dd \overline z  \quad\text{and}\quad \omega_2:=(1+\vep)\alpha \,i \dd z\wedge \dd \overline z .$$

Consider the upper envelop
$$ \Psi_1:=  \sup_\phi \big\{  \phi \text{ is } \omega\text{-subharmonic}: \; \phi\leq 0  \text{ on } X,\; \phi\leq (1+\vep)U_{\sigma_r} \text{ on } \partial\B(x,r)  \big\},$$
where $U_{\sigma_r}$ is defined in \eqref{defn-Phir}.  Let $$\sigma_1:=\ddc\Psi_1+\omega.$$

\medskip

\noindent \textbf{Claim:} $\ddc \Psi_1=-\omega$ on $\B(x,r)$. In particular, $\sigma_1\in \cM(X\setminus\B(x,r))$.

\proof[Proof of Claim] 
By Lemma \ref{lem-upenvop}, it is enough to show $\Psi_1\leq (1+\vep)U_{\sigma_r}$ on $\B(x,r)$, where the second function has only one zero in $\B(x,r)$.  From \eqref{defn-Phir}, we have $$\ddc (1+\vep)U_{\sigma_r}=-(1+\vep)\alpha \,i \dd z\wedge \dd \overline z=-\omega_2 \quad\text{on} \quad \B(x,r),$$
which gives
$$\ddc \big(\Psi_1- (1+\vep)U_{\sigma_r}\big)\geq -\omega +\omega_2 \geq 0  \quad\text{on}\quad \B(x,r).$$
Lemma \ref{lem-upenvop} gives $\Psi_1= (1+\vep)U_{\sigma_r}$ on $\partial\B(x,r)$.
By maximal modulus principle, 
$$\Psi_1\leq  (1+\vep)U_{\sigma_r} \quad\text{on}\quad \B(x,r).$$
This proves the claim.   
\endproof

Now consider the function (letting $|z|:=\dist(z,x)$)
$$
\Psi_2 (z)=
\begin{cases}   \alpha\pi (-2\vep r^2-|z|^2+\vep|z|^2  )         &\quad\text{for}\quad  z\in \B(x,r)\\
 \alpha\pi \Big(2(1-\vep)R^2\log{|z|\over r} -2\vep r^2 -|z|^2+\vep|z|^2\Big)   &\quad\text{for}\quad  z\in \B(x,R)\setminus \B(x,r)\\
 0 &\quad\text{for}\quad  z\in X\setminus \B(x,R),
\end{cases}
$$
where $R>r$ is determined by the equation
\begin{equation}\label{equa-R}
{R^2\over r^2} \Big(\log {R\over r}-{1\over 2}\Big)={\vep\over 1-\vep}   .
\end{equation}

It is not hard to check that $\Psi_2$ is continuous. On  $\B(x,R)\setminus \partial\B(x,r)$, we have $$\ddc \Psi_2= -\alpha (1-\vep)\,i \dd z\wedge \dd \overline z =-\omega_1.$$
An easy computation of the derivative with respect to $|z|$  gives
$$\ddc \Psi_2 > 0 \,\text{ on }\,  |z|=r  \quad\text{and}\quad \ddc \Psi_2 = 0 \,\text{ on }\,  |z|=R.$$ 
In particular, $\Psi_2$ is an $\omega$-subharmonic function on $X$, and it is the $\omega$-potential of the probability measure $$\sigma_2:=\ddc\Psi_2+\omega.$$  Furthermore, from the definition of $\Psi_1$ and that $\Psi_2=-\alpha\pi(1+\vep)r^2=(1+\vep)U_{\sigma_r}$ on $\partial\B(x,r)$, we see that  $\Psi_2\leq \Psi_1$ on $X$.   Applying Lemma \ref{lem-compare-oI}, we get
$$\oI_{\omega,r}(\sigma_1 )\leq  \oI_{\omega}(\sigma_2 ).$$
Note that $\supp(\sigma_2)$ may not be outside $\B(x,r)$.

\smallskip
To prove the lemma, it suffices to show  $ \oI_\omega(\sigma_2)\leq \alpha^2 e^2  \pi^2 r^4+O(r^5)$ since $\min  \oI_{\omega,r}\leq\oI_{\omega,r}(\sigma_1 )$.
In the following, we will estimate $\oI_{\omega,r}(\sigma_2 )$.
By definition \eqref{functional I_r} and \eqref{ineq-omega},
\begin{align*}
\oI_{\omega}(\sigma_2)  &= -\int_X \Psi_2 \, \omega - \int_X \Psi_2 \, \dd\sigma_2 = -\int_{\B(x,R)} \Psi_2 \, \omega - \int_{\B(x,R)} \Psi_2 \, \dd\sigma_2 \\
&\leq -\int_{\B(x,R)} \Psi_2 \,\omega_2-\int_{\B(x,R)} \Psi_2 \, (\ddc\Psi_2+\omega_2) \\
&=-\Big(1+{2\vep\over 1+\vep}\Big)\int_{\B(x,R)} \Psi_2 \,\omega_2 -\int_{\B(x,R)} \Psi_2 \, (\ddc\Psi_2+\omega_1).
\end{align*}
We  compute the exact value of the two integrals in coordinate $z$ as follows:
\begin{align*}
&-\int_{\B(x,R)} \Psi_2 \,\omega_2=-\int_{\B(x,r)} \Psi_2 \,\omega_2 -\int_{\B(x,R)\setminus \B(x,r)} \Psi_2 \,\omega_2 \\
&= \alpha\pi \int_{|z|<r} (2\vep r^2+|z|^2-\vep|z|^2)  (1+\vep)\alpha \,i \dd z\wedge \dd \overline z +  \\
   & \qquad \alpha\pi\int_{r<|z|<R} \Big(2(\vep-1)R^2\log{|z|\over r} +2\vep r^2 +|z|^2-\vep|z|^2\Big) (1+\vep)\alpha \,i \dd z\wedge \dd \overline z  \\
&=\alpha^2\pi(1+\vep) \int_0^{2\pi}\int_0^r (2\vep r^2+t^2-\vep t^2 )2t \,\dd t\dd\theta+  \\   
&\qquad \alpha^2\pi(1+\vep)\int_0^{2\pi}\int_r^R  \Big(2(\vep-1)R^2\log{t\over r} +2\vep r^2 +t^2-\vep t^2\Big)  2t \,\dd t\dd \theta        \\
&=\alpha^2\pi(1+\vep) \cdot 2\pi(1-\vep)(R^4/2-r^2R^2),
\end{align*}
where we have substituted \eqref{equa-R} to simplify.

For the measure $\ddc\Psi_2+\omega_1$ on $\B(x,R)$, it only has mass on $\partial \B(x,r)$, and the mass is equal to the area of $\omega_1(\B(x,R))$  due to Stoke's formula.   Hence
$$-\int_{\B(x,R)} \Psi_2 \, (\ddc\Psi_2+\omega_1)= \alpha \pi( 1+\vep )r^2\cdot 2\pi \alpha(1-\vep) R^2.  $$
Combining all the estimates above, yields
$$ \oI_\omega(\sigma_2)\leq \alpha^2\pi^2 (1+\vep)(1-\vep)R^4 .  $$

Recall that $\vep=2\rho r$ and from \eqref{equa-R}, we see that $R=\sqrt e \,r +O(r^2)$ as $r\to 0$. Therefore,
$$ \oI_\omega(\sigma_2)\leq \alpha^2e^2 \pi^2 r^4+O(r^5).$$
This completes the proof of the lemma.
\end{proof}

We also have the following lower bound.

\begin{lemma}
Under condition \eqref{eqn-flat-3}, we have $$\min\oI_{\omega,r} \geq \alpha^2 e^2  \pi ^2 r^4 -O(r^5).$$
\end{lemma}

\begin{proof}
We put $\delta:=\rho r$ to simplify the notation, where $\rho$ is the constant in \eqref{ineq-omega-z}. Fix a large positive number $\kappa>3$, whose value will be determined later.
Let $\widetilde\omega$ be a new smooth K\"ahler form  on $X$ satisfying:
\begin{equation}\label{con-wide-omega}
 \int_X \widetilde \omega=1   \quad\text{and}\quad     \begin{cases}    \widetilde \omega=(1+\kappa \delta)\alpha \,i \dd z\wedge \dd \overline z   &\text{ on }\, \B(x,2r)\\
 \widetilde \omega\leq (1+\kappa \delta)\alpha \,i \dd z\wedge \dd \overline z   &\text{ on }\, \B(x,3r)\\
 \widetilde\omega=  \omega &\text{ on }\, \B(x,r_0)\setminus\B(x,3r)\\
 \widetilde\omega\geq   \omega      & \text{ on }\, \B(x,r_0).
\end{cases}   
\end{equation}
 The existence of such $\widetilde\omega$ is guaranteed by \eqref{ineq-omega-z}. 

We only consider $r$ small enough such that  $U_{\nu_{\omega,r}}=0$ on $X\setminus \B(x,r_0)$.  Let $\widetilde\Psi_3$ be the unique solution of the following Dirichlet problem:
$$\widetilde\Psi_3=U_{\nu_{\omega,r}}  \,\text{ on }\, \partial\B(x,r)     \quad\text{and}\quad  \ddc  \widetilde\Psi_3=-\widetilde\omega  \,\text{ on }\, \B(x,r).  $$
Note that $\ddc((1+\kappa)\delta\alpha\pi|z|^2 )=(1+\kappa)\delta\alpha  \,i \dd z\wedge \dd \overline z$ and by \eqref{ineq-omega-z}, \eqref{con-wide-omega}, 
$$\ddc (\widetilde\Psi_3- U_{\nu_{\omega,r}})\geq -(1+\kappa)\delta\alpha  \,i \dd z\wedge \dd \overline z  \,\text{ on }\,   \B(x,r).$$
Applying maximal modulus principle to the function $\widetilde\Psi_3- U_{\nu_{\omega,r}}+(1+\kappa)\delta\alpha\pi|z|^2$, which is subharmonic on $\B(x,r)$, we get
$$       \widetilde\Psi_3- U_{\nu_{\omega,r}}   \leq (1+\kappa) \delta\alpha\pi r^2 \,\text{ on }\,   \B(x,r).       $$

\smallskip

Now consider the following function (letting $|z|:=\dist(z,x)$)
$$
\psi (z)=
\begin{cases}   \widetilde\Psi_3- U_{\nu_{\omega,r}}   - (1+\kappa) \delta\alpha\pi r^2         &\quad\text{for}\quad  z\in \B(x,r)\\
 \delta\alpha\pi \Big(2(\kappa-2) R^2 \log {|z|\over r} -3  r^2 -(\kappa-2)|z|^2 \Big)   &\quad\text{for}\quad  z\in \B(x,R)\setminus \B(x,r)\\
 0 &\quad\text{for}\quad  z\in X\setminus \B(x,R),
\end{cases}
$$
where $R$ is determined by the equation:
$$2(\kappa-2) R^2 \log {R\over r}= 3r^2+(\kappa-2)R^2.$$
Since $R\to \sqrt e \, r$ as $\kappa\to +\infty$,  we can fix a large $\kappa$ such that $\sqrt e\, r<R<2r$.

Observe that  $\psi$ is continuous, $ - (1+\kappa) \delta\alpha\pi r^2  \leq \psi\leq 0$  on $X$, and
\begin{equation}\label{mass-psi}
\ddc \psi=-\widetilde\omega +\omega \,\text{ on }\,\B(x,r),\quad \ddc \psi=-(\kappa-2)\delta\alpha  \,i \dd z\wedge \dd \overline z \,\text{ on }\, \B(x,R)\setminus \overline\B(x,r)
\end{equation}
Moreover, an easy computation of the derivative with respect to $|z|$, gives 
\begin{equation}\label{ddcpsi}
\ddc \psi > 0 \,\text{ on }\, |z|=r \quad\text{and}\quad \ddc \psi = 0 \,\text{ on }\, |z|=R.
\end{equation}

\medskip

\noindent\textbf{Claim:} $U_{\nu_{\omega,r}}+\psi$ is $\widetilde\omega$-subharmonic on $X$.

\proof[Proof of Claim] 
We need to show $\ddc (U_{\nu_{\omega,r}}+\psi)\geq -\widetilde\omega$ on $X$. On $\B(x,r)$, we have
$$\ddc (U_{\nu_{\omega,r}}+\psi)= -\omega-\widetilde\omega+\omega=-\widetilde\omega.$$
On $\B(x,R)\setminus \overline\B(x,r)$, using that $R<2r$ and \eqref{ineq-omega-z} we have
\begin{align*}
\ddc (U_{\nu_{\omega,r}}+\psi) &=-\omega  -(\kappa-2)\delta\alpha  \,i \dd z\wedge \dd \overline z  \\
&\geq   -(1+2\delta)\alpha  \,i \dd z\wedge \dd \overline z-(\kappa-2)\delta\alpha  \,i \dd z\wedge \dd \overline z\\
&= -( 1+\kappa    \delta)\alpha  \,i \dd z\wedge \dd \overline z=-\widetilde\omega.  
\end{align*}
On $\B(x,r_0)\setminus \overline \B(x,R)$, from \eqref{con-wide-omega}, we see that
$$ \ddc (U_{\nu_{\omega,r}}+\psi)=\ddc U_{\nu_{\omega,r}}\geq -\omega\geq -\widetilde\omega.  $$
On $X\setminus \B(x,r_0)$, recall that we assume $U_{\nu_{\omega,r}}=0$ there, where we have
$$\ddc (U_{\nu_{\omega,r}}+\psi) =\ddc U_{\nu_{\omega,r}}=0\geq -\widetilde\omega. $$
It remains to check the cases $|z|=r$ and $|z|=R$. This follows by \eqref{ddcpsi} and the fact $-\omega\geq -\widetilde\omega$ there. We finish the proof of the claim.
\endproof

Observe that $U_{\nu_{\omega,r}}+\psi$ is the $\widetilde\omega$-potential of the probability measure $$\eta:=\ddc (U_{\nu_{\omega,r}}+\psi)+\widetilde\omega,$$
whose support is outside $\B(x,r)$.
Since $\widetilde\omega$ is flat on $\B(x,2r)$, applying Proposition \ref{prop-2r} to $\widetilde\omega$ instead of $\omega$, we get
$$\oI_{\widetilde\omega, r}(\eta)\geq \min\oI_{\widetilde\omega,r} =e^2(1+\kappa\delta)^2\alpha^2\pi^2 r^4 =\alpha^2 e^2 \pi^2 r^4+O(r^5).       $$
By Lemma \ref{lem-diff-nu-eta} below, we finish the proof of the lemma.
\end{proof}

We put all the tedious computations below.

\begin{lemma}\label{lem-diff-nu-eta}
As $r\to 0$, 
$$  \big|\min\oI_{\omega,r} - \oI_{\widetilde\omega, r}(\eta)\big|=O(r^5) $$
\end{lemma}
\begin{proof}
By definition \eqref{functional I_r}, $\min\oI_{\omega,r} - \oI_{\widetilde\omega, r}(\eta)$ is equal to
$$ \int_X  (U_{\nu_{\omega,r}}+\psi) \,\widetilde\omega-\int_X  U_{\nu_{\omega,r}}  \,\omega + \int_X (U_{\nu_{\omega,r}}+\psi) \,\dd \eta  -   \int_{X}  U_{\nu_{\omega,r}}  \,\nu_{\omega,r}.  $$

By Lemmas \ref{lem-Util-omega} and \ref{lem-int-psi} below,
\begin{align*}
\Big|\int_X  (U_{\nu_{\omega,r}}+\psi) \,\widetilde\omega-\int_X  U_{\nu_{\omega,r}}  \,\omega  \Big| 
    \leq  \Big|\int_{X}  U_{\nu_{\omega,r}} \,(\widetilde\omega-\omega )\Big|+\Big| \int_{X} \psi \,\widetilde\omega    \Big|= O(r^5). 
\end{align*}

To bound another difference, we write
\begin{align*}
&\int_X (U_{\nu_{\omega,r}}+\psi) \,\dd \eta =  \int_X (U_{\nu_{\omega,r}}+\psi) \big(\ddc (U_{\nu_{\omega,r}}+\psi)+\widetilde\omega \big) \\
&= \int_X U_{\nu_{\omega,r}} \,\dd \nu_{\omega,r} +2 \int_X  \psi  \,  \ddc U_{\nu_{\omega,r}} + \int_X \psi  \,\ddc\psi+\int_X  \psi \,\widetilde\omega+ \int_X U_{\nu_{\omega,r}}  ( \widetilde\omega-\omega )
\end{align*}
where we use Stoke's formula to identify $\int_X  \psi \,  \ddc U_{\nu_{\omega,r}} $ with $\int_X  U_{\nu_{\omega,r}}\ddc\psi$.
Therefore,
\begin{align*}
\Big| \int_X (U_{\nu_{\omega,r}}+\psi) \,\dd \eta  &-   \int_{X}  U_{\nu_{\omega,r}}  \,\nu_{\omega,r}    \Big|    \\  
&\leq 2\Big|  \int_X  \psi  \,  \ddc U_{\nu_{\omega,r}}\Big| + \Big|\int_X \psi  \,\ddc\psi\Big|+\Big|\int_X  \psi \,\widetilde\omega \Big|+ \Big|\int_X U_{\nu_{\omega,r}}   ( \widetilde\omega-\omega )\Big|.
\end{align*}
Last sum is $O(r^5)$ by Lemmas   \ref{lem-Util-omega}, \ref{lem-int-psi} and \ref{ineq-Unu}  below. The result follows.
\end{proof}

\begin{lemma}\label{lem-Util-omega}
As $r\to 0$,
$$\Big|\int_{X}  U_{\nu_{\omega,r}} \,(\widetilde\omega-\omega )\Big|=O(r^5).$$
\end{lemma}

\begin{proof}
Remind that $\omega=\widetilde\omega$ on $\B(x,r_0)\setminus \B(x,3r)$ and $\supp(U_{\nu_{\omega,r}})\subset \B(x,r_0)$, implying
$$  \int_{X}  U_{\nu_{\omega,r}} \,(\widetilde\omega-\omega )=  \int_{\B(x,3r)}  U_{\nu_{\omega,r}} \,(\widetilde\omega-\omega ).$$

From \eqref{ineq-omega-z} and \eqref{con-wide-omega}, we see that  
$$ |\omega-\widetilde\omega|\leq  ( \kappa+3 )\delta     \alpha \,i \dd z\wedge \dd \overline z  \quad \text{on }\, \B(x,3r)   $$
and
$$      (1+3\delta)^{-1} \omega  \leq   \alpha \,i \dd z\wedge \dd \overline z \leq  (1-3\delta)^{-1} \omega       \quad \text{on }\, \B(x,3r).  $$
Since $U_{\nu_{\omega,r}}$ is non-positive, Lemma \ref{lem-lower bound} gives
\begin{equation}\label{ineq-Unu}
\int_{X}  |U_{\nu_{\omega,r}}|  \,\omega < \int_{X}  |U_{\nu_{\omega,r}}|  \,\omega  +  \int_{X}  |U_{\nu_{\omega,r}}|  \,\nu_{\omega,r} =\oI_{\omega,r}(\nu_{\omega,r}) \leq \alpha^2 e^2 \pi^2r^4 +O(r^5).     
\end{equation}
Therefore,
\begin{align*}
\Big|\int_{\B(x,3r)}  U_{\nu_{\omega,r}} \,(\widetilde\omega-\omega ) 
&\leq   \int_{\B(x,3r)}  |U_{\nu_{\omega,r}}| (\kappa +3)\delta \alpha  \,i \dd z\wedge \dd \overline z\\
&\leq {(\kappa +3)\delta\over 1-3\delta} \int_{\B(x,3r)}  |U_{\nu_{\omega,r}}|  \,\omega
\leq 2\kappa \delta\int_{X}  |U_{\nu_{\omega,r}}|  \,\omega=O(r^5).
\end{align*}
This proves the lemma..
\end{proof}

\begin{lemma}\label{lem-int-psi}
As $r\to 0$,
$$\Big| \int_{X} \psi \,\widetilde\omega    \Big|=O(r^5).$$
\end{lemma}

\begin{proof}
Recall that $\supp(\psi)\subset  \B(x,2r),$, the above integral is actually integrating over $\B(x,2r)$. Using $- (1+\kappa) \delta\alpha\pi r^2  \leq \psi\leq 0$  on $X$,
we have
\begin{align*}
\Big|\int_{\B(x,2r)} \psi \,\widetilde\omega    \Big|  & =
 \Big|\int_{\B(x,2r)} |\psi| (1+\kappa \delta)\alpha \,i \dd z\wedge \dd \overline z\\
&\leq  (1+\kappa) \delta\alpha\pi r^2    (1+\kappa \delta)\alpha  \int_{\B(x,2r)} \,i \dd z\wedge \dd \overline z\leq  2\delta\alpha^2\pi r^2 O(r^2)=O(r^5).
\end{align*}
This finishes the proof.
\end{proof}

\begin{lemma}\label{lem-intpsiU}
As $r\to 0$,
$$   \Big|\int_X  \psi \,  \ddc U_{\nu_{\omega,r}}\Big|=O(r^5),\quad  \Big|\int_X  \psi \,  \ddc \psi\Big|=O(r^6)  ,\quad \Big|\int_X  U_{\nu_{\omega,r}}   \ddc U_{\nu_{\omega,r}}\Big| =O(r^4).     $$
\end{lemma}

\begin{proof}
By Cauchy-Schwarz inequality and Stoke's formula,
\begin{align*}
\Big|\int_X  \psi  \, \ddc U_{\nu_{\omega,r}}\Big| =\Big|\int_X  \dd\psi \wedge  \dc U_{\nu_{\omega,r}}\Big|&\leq \Big|\int_X  \dd\psi \wedge  \dc \psi\Big| ^{1/2} \cdot \Big|\int_X  \dd U_{\nu_{\omega,r}} \wedge  \dc U_{\nu_{\omega,r}}\Big|^{1/2} \\
&=\Big|\int_X  \psi   \, \ddc \psi \Big|^{1/2} \cdot \Big|\int_X  U_{\nu_{\omega,r}}   \ddc U_{\nu_{\omega,r}}\Big|^{1/2}.
\end{align*}
Thus, to prove the lemma, we only need to show the second and third equations. The third one is followed by \eqref{ineq-Unu}. For the second estimate, we first find an upper bound for the mass of the measure $|\ddc \psi|$, whose support is contained in $\overline \B(x,R)$. We already know it does not have mass on $\partial\B(x,R)$ by \eqref{ddcpsi}.  Therefore, by Stoke's fromula,
$$0= \int_X \ddc\psi =\int_{\B(x,R)\setminus \partial\B(x,r)} \ddc \psi  +\int_{\partial\B(x,r)}  \ddc \psi.$$
From \eqref{mass-psi}, it is not hard to see that 
$$ \Big|\int_{\B(x,R)\setminus \partial\B(x,r)} \ddc \psi  \Big| =O(r^3).  $$
So, the mass of $|\ddc\psi|$ is  $O(r^3)$ and 
$$\Big|\int_X  \psi  \,  \ddc \psi\Big| \leq \max|\psi|\cdot \int_X |\ddc \psi| \leq (1+\kappa) \delta\alpha\pi r^2 \cdot   O(r^3)=O(r^6).
$$
The proof of the lemma is finished.
\end{proof}

We conclude from this section that 

\begin{proposition}\label{prop-r1}
Under condition \eqref{eqn-flat-3}, as $r\to 0$, we have  
$$|\min\oI_{\omega,r} -\alpha^2 e^2  \pi ^2 r^4| =O(r^5).$$
\end{proposition}

\medskip
\section{{Perturb distance metric}}\label{sec-dis}

We are now ready to prove the main theorem.
In the general case, the open ball $\B(x,r)$ is not a Euclidean disc. We will use  a ``sandwich argument", finding  two  discs to bound it, where we  already know how to estimate the functional $\oI_{\omega,r}$. Finally,  the result will follow by the monotone property of $\oI_{\omega,r}$ on $r$.

\medskip

Fix an $r_0>0$ and a local coordinate $z$ on $\B(x,r_0)$ such that $z=0$ at $x$.  Since $\omega_0$ is smooth, we have near $x$, 
$$\omega_0(z)= (1+O(|z|) \beta \,i \dd z\wedge \dd \overline z      $$
for some $\beta>0$. In what follows, for $r$ small, we use $\B_{\omega_\C}(x,r)$ to denote the open ball of radius $r$ centered at $x$  with respect to the flat distance metric
$$\omega_\C:= 1/2 \, i \dd z\wedge \dd \overline z.   $$
We set $|z|:=\dist_{\omega_\C} (z,x)$.
There exists a constant $\varrho>0$ such that 
\begin{equation}\label{ineq-omega0-z}
(1-\varrho |z|)\beta   \,i \dd z\wedge \dd \overline z   \leq \omega_0 \leq (1+\varrho |z|) \beta \,i \dd z\wedge \dd \overline z \quad\text{on}\quad \B(x,r_0).
\end{equation}

\smallskip

\begin{proof}[Proof of Theorem \ref{thm-main-hole}]
For any point $y$ with $\dist_{\omega_\C}(x,y)=r$, by \eqref{ineq-omega0-z}, we have
$$\dist_{\omega_0}(x,y) \leq \int_{[x,y]} \sqrt {(1+\varrho |z|) \beta}  \,|\dd z|\leq  \int_{[x,y]} \sqrt {(1+\varrho r) \beta}  \,|\dd z|=r\sqrt {2(1+\varrho r) \beta}.  $$
It follows that 
$$\B_{\omega_\C}(x,r) \subset \B \big(x,r\sqrt {2(1+\varrho r) \beta} \big).$$

On the other hand, using \eqref{ineq-omega0-z} again, for any smooth curve $\Gamma$ with end points $x$ and $y$, the length of $\Gamma$ with respect to $\omega_0$ is bounded from below by
$$ \int_\Gamma    \sqrt {(1-\varrho |z|) \beta}  \,|\dd z|\geq \sqrt {(1-\varrho r) \beta} \int_\Gamma    |\dd z|\geq     r\sqrt {2(1-\varrho r) \beta}.         $$
This gives $\dist_{\omega_0}(x,y)\geq  r\sqrt {2(1-\varrho r) \beta}$ and hence,
$$\B\big(x,r\sqrt {2(1-\varrho r) \beta} \big)\subset \B_{\omega_\C}(x,r).$$

We conclude that 
$$ \B_{\omega_\C} \bigg(x,{r\over \sqrt {2(1+\varrho r) \beta}}     \bigg)   \subset \B(x,r)\subset  \B_{\omega_\C} \bigg(x,{r\over \sqrt {2(1-\varrho r) \beta}}     \bigg).           $$

By definition, $\min\oI_{\omega,r}$ is monotone increasing in $r$, and its value is independent of the choice of $\omega_0$. We deduce from  Proposition \ref{prop-r1} that 
$$    \alpha^2 e^2  \pi ^2 \Big({r\over \sqrt {2(1+\varrho r) \beta}}\Big)^4  - O(r^5) \leq\min   \oI_{\omega, r}\leq  \alpha^2 e^2  \pi ^2 \Big({r\over \sqrt {2(1-\varrho r) \beta}}\Big)^4  +O(r^5) .     $$
After simplifying the expression, we get
$$ \Big| \min   \oI_{\omega, r} - {\alpha^2\over 4\beta^2}  e^2  \pi^2 r^4 \Big| =O(r^5).      $$
This completes the proof of Theorem \ref{thm-main-hole} with $C_x=\alpha^2/(4\beta^2)$.
\end{proof}

\end{document}